\documentclass[10pt]{amsart}
\usepackage{amssymb,amsmath}
\usepackage{amsthm,mathabx}
\usepackage{qsymbols}
\usepackage{graphicx}
\usepackage[utf8]{inputenc}
\usepackage{appendix}

\newtheorem{thmintro}{Theorem}
\newtheorem{thm}{Theorem}[section]
\newtheorem{lem}[thm]{Lemma}
\newtheorem{prop}[thm]{Proposition}

\newtheorem{cor}[thm]{Corollary}
\newtheorem{defi}[thm]{Definition}
\newtheorem{rem}[thm]{Remark}

\newtheorem{remark}[thm]{Remark}

\numberwithin{equation}{section}

\newcommand{\R}{\mathbb{R}}

\newcommand{\E}{\mathbb{E}}

\newcommand{\ep}{\varepsilon}
\newcommand{\ka}{\kappa}

\newcommand{\lc}{\leq_{\mathrm{cvx}}}

\newcommand{\bmu}{\bar{\mu}}

\renewcommand{\ep}{\varepsilon}
\newcommand{\CAT}{\mathrm{CAT}}

\DeclareMathOperator*{\argmin}{arg\,min}

\usepackage{color}
\definecolor{revcolor}{rgb}{0.75,0,0.15}

\makeatletter
\newcommand{\flushleftheadings}{%
  \def\@secnumfont{\bfseries}%
  \def\section{\@startsection{section}{1}%
    \z@{1.1\linespacing\@plus.6\linespacing}{.55\linespacing}%
    {\normalfont\large\bfseries\raggedright}}%
  \def\subsection{\@startsection{subsection}{2}%
    \z@{.8\linespacing\@plus.5\linespacing}{.35\linespacing}%
    {\normalfont\bfseries\raggedright}}%
  \def\subsubsection{\@startsection{subsubsection}{3}%
    \z@{.5\linespacing\@plus.7\linespacing}{-.5em}%
    {\normalfont\bfseries}}%
}
\makeatother

\flushleftheadings

\begin{document}

\title{A Brenier--Strassen theorem on $\CAT(\kappa)$ spaces}

 \author{Nathael Gozlan}
 \address{Université Paris Cité, CNRS, MAP5, F-75006 Paris, France}
 \email{nathael.gozlan@u-pariscite.fr}

 \author{Hugo Malamut}
 \address{Université Paris Cité, CNRS, MAP5, F-75006 Paris, France}
 \email{hugo.malamut@u-pariscite.fr}

 \author{Shin-Ichi Ohta}
 \address{University of Osaka, Osaka 560-0043, Japan \& RIKEN Center for Advanced Intelligence Project (AIP), 1-4-1 Nihonbashi, Tokyo 103-0027, Japan}
\email{s.ohta@math.sci.osaka-u.ac.jp}



\date{\today}

\thanks{The first and second named author are supported by a grant of the Agence nationale de la recherche (ANR), Grant ANR-23-CE40-0017 (Project SOCOT).
The third named author is supported by the JSPS Grant-in-Aid for Scientific Research (KAKENHI) 22H04942, 24K00523, 24K21511, 26H01996, and by the JST CREST JPMJCR25Q2. He is also grateful to Universit\"at Bonn for its hospitality during his visit in Summer 2026, a part of this work was carried out there.
This research has been conducted within the FP2M federation (CNRS FR 2036).}

\begin{abstract}
    We extend the Brenier--Strassen theorem about projections in convex order to non-flat spaces with curvature bounded from above.
    Precisely, for probability measures $\mu,\nu$ of finite second moment on a complete separable $\CAT(0)$ space, we prove that $\mu$ admits a unique $W_2$-projection $\bar\mu$ to the set of probability measures dominated by $\nu$ in convex order.
    Moreover, the unique optimal coupling from $\mu$ to $\bar\mu$ is induced by a $1$-Lipschitz map, without any absolute-continuity assumption on $\mu$.
    Our proof identifies the projection problem with a weak optimal transport problem whose cost is the squared distance to the set of convex means.
    We also establish a localized version on $\CAT(\kappa)$ spaces with $\kappa>0$, where the optimal map is $1/2$-H\"older continuous.
    Finally, we give a Strassen-type characterization of one-step barycentric martingales on proper $\CAT(0)$ spaces.
\end{abstract}

\maketitle

\section*{Introduction}

The aim of this article is to generalize a result about convex order projections in the Wasserstein space obtained in \cite{GJ20} to a non-Euclidean framework.

For a complete separable metric space $(E,d)$, we will denote by $\mathcal{P}_p(E)$, $p \geq 1$, the set of Borel probability measures $\mu$ such that $\int_E d^p(x_0,x) \,d\mu(x)<+\infty$ for some (and thus all) $x_0\in E$.
For $\mu,\nu \in \mathcal{P}_2(E)$, the quadratic Monge--Kantorovich distance, also known as Wasserstein distance, denoted by $W_2$, is defined as
\begin{equation}\label{eq:W2-intro}
  W_2^2(\mu,\nu) := \inf_{\pi \in \Pi(\mu,\nu)} \int_{E \times E} d^2(x,y)\,\pi(dxdy), \qquad \forall \mu,\nu \in \mathcal{P}_2(E),
\end{equation}
where $\Pi(\mu,\nu)$ is the set of all transport plans from $\mu$ to $\nu$, that is, the set of all probability measures $\pi$ on $E\times E$ such that $\pi(A\times E)=\mu(A)$ and $\pi(E\times A)=\nu(A)$ for all Borel sets $A \subset E$.
As is well known, there always exist transport plans realizing the infimum in \eqref{eq:W2-intro}; such plans are called optimal transport plans. The distance $W_1$ is defined analogously on $\mathcal{P}_1(E)$, with $d$ in place of $d^2$.

When $E=\R^d$ is equipped with the standard Euclidean distance $d$, the transport problem above is very well understood.
According to a classical result of Brenier \cite{Brenier1,Brenier2}, whenever $\mu$ is absolutely continuous, there is a unique optimal transport plan $\pi^*$ for the problem \eqref{eq:W2-intro} which is of the form $\pi^* = \mathrm{Law}(X,\nabla\phi(X))$ with $X$ a random variable of law $\mu$ and $\phi:\R^d \to \R \cup \{+\infty\}$ a convex function.
The map $T = \nabla \phi$ is well defined $\mu$-almost everywhere and is called the Brenier transport map from $\mu$ to $\nu$.
Regularity properties of the Brenier map have attracted a lot of attention; see in particular \cite{Caf92,Caffarelli2000}.
As concerns existence of the Brenier map, the absolute continuity assumption on $\mu$ can be relaxed a bit; see \cite{McCann95,GMC96,Gigli11}.
It is however easy to construct examples where there is no transport map, or examples where there are transport maps, but no optimal ones.
There are also well known examples where the Brenier map exists but is not continuous.

In \cite{GJ20}, a simple geometric condition has been identified for the existence and the Lipschitz continuity of an optimal transport map.
To recall this result, let us introduce some additional notation.
Two probability measures $\mu,\nu \in \mathcal{P}_1(\R^d)$ are said to be in convex order, denoted by $\mu \lc \nu$, if $\int_{\R^d} f\,d\mu \leq \int_{\R^d} f\,d\nu$ for all convex functions $f:\R^d \to \R$.
If $\mu,\nu \in \mathcal{P}_2(\R^d)$, there exists a unique probability measure $\bar{\mu} \lc \nu$ such that
\begin{equation}\label{eq:barmu-intro}
 W_2^2(\mu,\bar{\mu}) = \inf_{\eta \lc \nu} W_2^2(\mu,\eta).
\end{equation}
In other words, $\bar{\mu}$ is the unique metric projection (also called the nearest point projection) of $\mu$ to the convex set $\{\eta \in \mathcal{P}_2(\R^d) : \eta \lc \nu\}$ of all probability measures dominated by $\nu$ in convex order.
Existence and uniqueness of $\bar{\mu}$ were obtained in \cite{GJ20} or \cite{ACJ20}.
This convex order projection has been further studied by several authors \cite{Kim-Ruan,backhoff2020weak,Alfonsi-Jourdain-preprint}.
It is naturally related to the weak optimal transport problem with a quadratic barycentric cost \cite{GRST17,ABC19}, and admits interesting applications in terms of concentration of measure \cite{GRST18}, sampling \cite{ACJ20}, and extrapolation of geodesics \cite{GNT25}.
More importantly for the purpose of this article, it was shown in \cite{GJ20} that the optimal transport between $\mu$ and its projection $\bar{\mu}$ is always deterministic (that is, given by a transport map) and regular.
Precisely, there exists a continuously differentiable convex function $\phi : \mathbb{R}^d \to \mathbb{R}$ such that the map $T = \nabla \phi$ pushes $\mu$ forward to $\bar{\mu}$ and is $1$-Lipschitz.
The converse is also true: if $\phi$ is a differentiable convex function transporting a probability measure $\mu$ to another probability measure $\nu$, and if $\nabla \phi$ is $1$-Lipschitz, then $\bar{\mu} = \nu$.
See \cite{FGP20} for a proof of Caffarelli's contraction theorem based on this idea.

We shall generalize this construction to $\CAT(\kappa)$ spaces with $\kappa\geq 0$.
Precise definitions will be recalled in Section \ref{sec:cat}.
The class of $\CAT(0)$ spaces contains Hilbert spaces and their closed convex subsets, metric trees, Euclidean buildings, and complete simply-connected Riemannian manifolds with nonpositive sectional curvature (in particular, hyperbolic spaces).
We refer to \cite{SturmNPC,BH99,Bacak} for an overview of the subject.

The case $\kappa>0$ is treated in Theorems~\ref{thm:main-WOT}, \ref{thm:BS-CATkappa} below.
In the case $\kappa =0$, we obtain the following result.

\begin{thmintro}\label{thm:BS-NPC}
Suppose that $(E,d)$ is a complete separable $\CAT(0)$ space.
For any $\mu,\nu \in \mathcal{P}_2(E)$, the following hold.
\begin{enumerate}
\item There exists a unique $\bar{\mu} \lc \nu$ such that $\inf_{\eta \lc \nu} W_2^2(\mu,\eta) = W_2^2(\mu,\bar{\mu})$.
\item There exists a unique $W_2$-optimal transport plan from $\mu$ to $\bar{\mu}$ given by a $1$-Lipschitz transport map $T$ defined on the support of $\mu$.
\end{enumerate}
\end{thmintro}

Theorem \ref{thm:BS-NPC} is deduced from the following structure theorem, which recasts the projection problem \eqref{eq:barmu-intro} as a weak optimal transport problem and provides a full description of its solutions.
Below $C(q)$ denotes the set of all convex means of a probability measure $q$; that is, the set of points $x\in E$ such that $f(x) \leq \int_E f\,dq$ for all lower semicontinuous convex functions $f:E\to\R \cup \{+ \infty\}$ (see Section \ref{sec:convex-means}), and $\mu p = \nu$ means that a probability kernel $p=(p_x)_{x\in E}$ satisfies $\int_E p_x(dy) \,\mu(dx)=\nu(dy)$.

\begin{thmintro}\label{thm:projection_WOT}
Let $(E,d)$ be a complete separable $\CAT(0)$ space and $\mu,\nu \in \mathcal{P}_2(E)$.
Then
\[
\min_{\eta \lc \nu} W_2^2(\mu,\eta) \;=\; \min_{p:\mu p = \nu} \int_E d^2\bigl(x, C(p_x)\bigr)\,\mu(dx) \;=\; \max_{f:\text{convex}} \left\{ \int_E Qf\,d\mu - \int_E f\,d\nu \right\},
\]
where the second minimum is taken over all probability kernels $p$ such that $\mu p = \nu$, the maximum over all $\nu$-integrable, lower semicontinuous, convex functions $f:E \to \R \cup \{+\infty\}$, and $Qf(x) := \inf_{y\in E} \{f(y)+d^2(x,y)\}$.
Solutions $\bar \mu,p,f$ of these problems are related by a unique map $T$ defined $\mu$-almost everywhere as follows:
\begin{enumerate}
\item $T$ is the $W_2$-optimal transport map from $\mu$ to $\bar \mu$;
\item For $\mu$-almost all $x$, $T(x)$ is the metric projection of $x$ to $C(p_x)$;
\item $T$ is the proximal operator\footnote{Our definition of proximal operator differs from the usual one by a factor $1/2$.} associated with $f$ in the sense that, for $\mu$-almost all $x$,
\[ T(x) = \argmin_{y \in E} \Bigl\{f(y) + d^2(x,y)\Bigr\}.
\]
\end{enumerate}
\end{thmintro}

The optimal transport map $T$ from $\mu$ to $\bar\mu$ describes the first step of a coupling of $\mu$ and $\nu$ provided by an optimal weak transport plan.
As in the Euclidean setting, the second step is a coupling between $\bar\mu$ and $\nu$ which satisfies a martingale property in the sense of Émery and Mokobodzki \cite{EM91}, but in general not in the (stronger) sense of Sturm \cite{Sturm02} (see Remark \ref{rem:martingale} below for a precise statement).

The paper is organized as follows.
In Section~\ref{sec:cat}, we recall the definition of $\CAT(\kappa)$ spaces with $\kappa \geq 0$, of convex order between probability measures, and the notions of barycenter and convex mean.
In Section~\ref{sec:wot}, we study the weak optimal transport problem associated with the cost function $c(x,p) = d^2(x,C(p))$.
The first part of the section is devoted to continuity properties of this cost with respect to the product topology on $E\times\mathcal{P}_2(E)$, and the second part to the proof of a general version of Theorem~\ref{thm:projection_WOT}, valid on $\CAT(\kappa)$ spaces (Theorem~\ref{thm:main-WOT}).
In Section~\ref{sec:regularity}, we study the regularity of the optimal transport map $T$ from $\mu$ to its projection $\bar{\mu}$ and complete the proof of Theorem~\ref{thm:BS-NPC} as well as its generalization to $\CAT(\kappa)$ spaces with $\kappa>0$ (Theorem~\ref{thm:BS-CATkappa}).
Finally, in Section~\ref{sec:barycentric}, we apply similar ideas to give a Strassen type characterization for barycentric martingales in proper $\mathrm{CAT}(0)$ spaces.

\section{$\CAT(\kappa)$ spaces and convexity}\label{sec:cat}

Throughout the section, $(E,d)$ will always be a complete separable geodesic space. We recall that the product space $E\times E$ equipped with the $\ell_2$ distance
\[
d_2\bigl((a,b),(c,d)\bigr) = \sqrt{d^2(a,c)+d^2(b,d)}, \qquad (a,b), (c,d) \in E\times E,
\]
is then also a complete separable geodesic space.

\subsection{$\CAT(\kappa)$ spaces}

\subsubsection{Definition}

Let $M_0^2 := \R^2$ and, for $\kappa > 0$, let $M_\kappa^2$ denote the 2-sphere of radius $1/\sqrt{\kappa}$.
Define $d_\kappa$ as the corresponding geodesic distance on $M_\kappa^2$.
We set
\[
D_\kappa:=
\begin{cases}
+\infty, & \kappa=0,\\
\pi/\sqrt{\kappa}, & \kappa>0.
\end{cases}
\]

\begin{defi}
A geodesic metric space $(E,d)$ is called a \emph{$\CAT(\kappa)$ space} if every geodesic triangle\footnote{A geodesic triangle $\Delta(x,y,z)$ is the union $[x,y]\cup [y,z]\cup [z,x]$ of three geodesic segments; its perimeter is the sum of the lengths of the segments.} of perimeter $<2D_\kappa$ is no thicker than its comparison triangle in $M_\kappa^2$.
More precisely, if $\Delta(x,y,z)\subset E$ is a geodesic triangle with comparison triangle\footnote{A comparison triangle $\Delta(\bar x,\bar y,\bar z)\subset M_\kappa^2$ is a geodesic triangle in the model surface such that $d(x,y)=d_\kappa(\bar x,\bar y)$, $d(y,z)=d_\kappa(\bar y,\bar z)$, $d(z,x)=d_\kappa(\bar z,\bar x)$.} $\Delta(\bar x,\bar y,\bar z) \subset M_\kappa^2$, and $\bar p,\bar q$ are the comparison points\footnote{If $p\in [x,y]$, for instance, the comparison point $\bar p\in [\bar x,\bar y]$ is chosen so that $d(x,p)=d_\kappa(\bar x,\bar p)$.} of $p,q\in \Delta(x,y,z)$, then
\[
d(p,q)\le d_\kappa(\bar p,\bar q).
\]
\end{defi}

For broader background on spaces with curvature bounded above we refer to \cite{BH99}.

\subsubsection{Examples}

When $\kappa=0$, one recovers the class of nonpositively curved spaces in the sense of Alexandrov, also called global NPC spaces or Hadamard spaces (when complete).
In that case, an equivalent definition is the following property of the squared distance:
for any geodesic $(x_t)_{t\in[0,1]}$ and any point $z$, we have
\begin{equation}\label{eq:CAT0}
d^2(z,x_t) \leq (1-t)\, d^2(z,x_0)+t\, d^2(z,x_1)-t(1-t)\, d^2(x_0,x_1), \qquad \forall t\in [0,1].
\end{equation}
Typical examples are Hilbert spaces, metric trees, Euclidean buildings, Hadamard manifolds (complete simply connected Riemannian manifolds with nonpositive sectional curvature), as well as products and closed convex subsets of such spaces.
For the properties of such spaces, we refer to \cite{BH99,SturmNPC}.

For $\kappa>0$, the basic model example is the round sphere $S^n_{1/\sqrt{\kappa}}$, and more generally geodesically convex subsets of the sphere and spherical buildings provide natural examples.
We refer again to \cite{BH99} for details and examples.

\subsection{Convex functions}

We recall that a subset $D \subset E$ is said to be (\emph{geodesically}) \emph{convex} if, for all $x_0,x_1 \in D$ and every geodesic $(x_t)_{t\in [0,1]}$ joining $x_0$ to $x_1$, we have $x_t\in D$ for all $t\in [0,1]$.
A function $f:D \to \mathbb{R}\cup \{+ \infty \}$ defined on a convex subset $D$ of a geodesic metric space $(E,d)$ is \emph{convex} if for all geodesics $(x_t)_{t\in [0,1]}$ valued in $D$, it holds
\begin{equation}\label{eq:conv}
 f(x_t) \leq (1-t) f(x_0)+tf(x_1),\qquad \forall t\in (0,1).
\end{equation}
One can slightly relax these conditions by requiring merely the existence of a geodesic satisfying the above properties (sometimes called weak convexity), however, it makes no difference in what follows since we consider the case where geodesics in question are unique.

If $D \subset E$ is a closed convex set, then a convex function $f:D \to \mathbb{R}\cup \{+ \infty \}$ canonically extends to a convex function defined on $E$ by simply taking value $+ \infty$ outside of $D$; we might also call this function $f$ in the following.
The same applies to lower semicontinuous (l.s.c.\ in short) functions.

The convexity is strict if the inequality \eqref{eq:conv} is strict for $t \in (0,1)$ and $x_0 \neq x_1$.
For $k\in \mathbb{R}$, $f$ is called \emph{$k$-convex} if we have
\[
f(x_t)\le (1-t)f(x_0)+t\,f(x_1)-\frac{k}{2}\,t(1-t)\,d^2(x_0,x_1),
\qquad \forall t\in(0,1),
\]
instead of \eqref{eq:conv}.
The simplest function to exhibit $k$-convexity is the squared distance:
inequality \eqref{eq:CAT0} expresses precisely that, in $\CAT(0)$ spaces, the function $d^2(z,\,\cdot\,)$ is $2$-convex on $E$, for every $z\in E$.
When $\kappa > 0$, there is a local analog (see for example \cite[Proposition 3.1(i)]{Ohta}):

\begin{lem}\label{lem:Ohta}
Let $(E,d)$ be a $\CAT(\kappa)$ space with $\kappa>0$.
Recall $D_\kappa=\pi/\sqrt{\kappa}$.
Fix $\ep\in(0,D_\kappa/2)$ and $z\in E$.
Then the function $d^2(z,\,\cdot\,)$ is $k$-convex on the closed ball $\overline{B}(z,\ep)$, where
\[
k := \frac{2\sqrt{\kappa}\,\ep}{\tan(\sqrt{\kappa}\,\ep)}.
\]
That is to say, for every $x_0,x_1\in \overline{B}(z,\ep)$ and the geodesic $(x_t)_{t\in[0,1]}$ joining $x_0$ to $x_1$, we have
\[
d^2(z,x_t) \leq (1-t)\,d^2(z,x_0) + t\,d^2(z,x_1) - \frac{k}{2}\, t(1-t)\,d^2(x_0,x_1), \qquad \forall t\in[0,1].
\]
\end{lem}

\subsection{Convex order}

\subsubsection{Definition.}
For $\eta,\nu\in \mathcal{P}_1(E)$, we will say that $\eta$ is dominated by $\nu$ in \emph{convex order}, denoted by $\eta\lc \nu$, if
\begin{equation}\label{eq:OC}
\int_E f\,d\eta \leq \int_E f\,d\nu
\end{equation}
for all l.s.c\ convex functions $f:E \to \mathbb{R}\cup \{+\infty\}$ such that $\int_E [f]_-\,d\eta<+\infty$ and $\int_E [f]_-\,d\nu<+\infty$, where $[f]_-:=\max(-f,0)$.
If $\eta \lc \nu$, then a direct consequence of the definition of convex order is that for any closed convex set $D$, $\nu(D)=1$ implies $\eta(D)=1$.

\subsubsection{Test functions.}
The following lemma shows that, under reasonable assumptions, the functions used to check the convex order according to \eqref{eq:OC} can be assumed to be continuous.

\begin{lem}\label{lem:tech}
Let $(E,d)$ be a complete separable $\CAT(\kappa)$ space with $\kappa \geq 0$ and let $\eta,\nu \in \mathcal{P}_1(E)$.
\begin{enumerate}
\item If $\kappa=0$, then $\eta \lc \nu$ if and only if \eqref{eq:OC} holds for all Lipschitz convex functions $f$ bounded from below.
\item If $\kappa>0$ and if $\eta$ and $\nu$ are concentrated in a closed convex subset $D\subset E$ with $D \subset \overline{B}(x_0,\ep)$ for some $x_0 \in E$ and $\ep <D_\kappa/2$, then $\eta \lc \nu$ if and only if \eqref{eq:OC} holds for all convex, bounded, uniformly continuous functions $f:\overline{B}(x_0,\ep) \to \R$.
The same conclusion holds if $\eta,\nu \in \mathcal{P}_1(E\times E)$ are concentrated in $D\times D$.
\end{enumerate}
\end{lem}

\begin{proof}
(1)
Assume that the bound \eqref{eq:OC} holds for Lipschitz convex functions $f:E\to \R$ bounded from below.
Let $\varphi:E\to \mathbb{R}\cup\{+\infty\}$ be a l.s.c.\ convex function bounded from below.
Consider the sequence of functions $f_n:E \to \mathbb{R}$, $n\geq 1$, defined by $f_n(x):=\inf_{y\in E}\{\varphi(y)+nd(x,y)\}$, $x\in E$, $n\geq 1$.
It is well known that $f_n$ is $n$-Lipschitz continuous on $E$ and converges pointwise monotonically to $\varphi$ from below, as $n\to \infty$.
Moreover, $f_n$ is also convex on $E$.
Indeed, the distance $d$ being jointly convex (in other words, convex in the sense of Busemann; see e.g.\ \cite[Corollary 2.5]{SturmNPC}), for two geodesics $(x_t)_{t\in [0,1]}$ and $(y_t)_{t\in [0,1]}$, we get
\[
f_n(x_t) \leq \varphi(y_t)+nd(x_t,y_t) \leq (1-t)[\varphi(y_0)+nd(x_0,y_0)]+t[\varphi(y_1)+nd(x_1,y_1)]
\]
and so, taking the infimum over $y_0,y_1\in E$, gives
\[
f_n(x_t) \leq (1-t)f_n(x_0)+tf_n(x_1).
\]
By assumption, we have
\[
\int_E f_n\,d\eta \leq \int_E f_n\,d\nu,
\]
thus, by monotone convergence, letting $n\to \infty$ yields
\[
\int_E \varphi\,d\eta \leq \int_E \varphi\,d\nu.
\]
The assumption that $\varphi$ is bounded from below can be removed by considering $\max(\varphi,-k)$ (which is still convex) and letting $k\to \infty$.
So we conclude that it is enough to have \eqref{eq:OC} for all Lipschitz convex functions $f$ which are bounded from below to ensure that $\eta \lc \nu$.

(2)
The proof is similar, so we only point out main differences and we deal with the case of measures $\eta, \nu \in \mathcal{P}_1(E\times E)$ which is more general.

Set $B:=\overline{B}(x_0,\ep)$ with $\ep <D_\kappa/2$.
We use an approximation procedure from \cite{YokotaCAT1}.
Given a l.s.c.\ convex function $\varphi:D\times D\to \R$ bounded from below, the sequence $f_n$ is now defined by
\[
f_n(x,y):=\inf_{(u,v) \in D\times D} \{\varphi(u,v)+n\Phi(x,u)+n\Phi(y,v)\},\qquad (x,y)\in B^2,\ n\geq 1,
\]
where $\Phi:B\times B \to \R$ is a jointly convex function satisfying
\begin{equation}\label{eq:ineq-Yokota}
c_0\,d^m (x,y) \leq \Phi(x,y) \leq C_0\,d^m(x,y),\qquad \forall x,y\in B,
\end{equation}
for some $c_0,C_0>0$ and $m>1$.
The existence of such a function $\Phi$ is granted by \cite[Theorem A]{YokotaCAT1} (see also (18) in \cite{YokotaCAT1}), which extends to $\CAT(\kappa)$ spaces with $\kappa>0$ a construction going back to \cite{MR535705,Kendall}.
Reasoning exactly as in the first part of the proof, we see that $f_n$ is convex on $B\times B$.
Moreover, it can be checked on the explicit formula given in \cite{YokotaCAT1} that $\Phi$ is uniformly continuous on $B\times B$, so that $f_n$ is uniformly continuous on $B\times B$ as well.
It is also easy to check that $f_n$ converges to $\varphi$ pointwise monotonically.
Then we conclude as in the first part of the proof.
\end{proof}

\subsubsection{Barycenters.}\label{sec:bary}
Let $(E,d)$ be a complete $\CAT(0)$ space.
We recall that, given $p \in \mathcal{P}_1(E)$, the \emph{barycenter} of $p$ is the unique minimizer of
\[
y \mapsto \int_E \{d^2(y,x)-d^2(o,x)\} \,p(dx),
\]
where $o \in E$ is an arbitrary fixed point.
We refer to \cite[Proposition 4.3]{SturmNPC} for existence and uniqueness.
The barycenter of $p$ does not depend on $o$, and will be denoted by $b(p)$ throughout what follows.
According to e.g.\ \cite[Theorem 6.2]{SturmNPC}, any l.s.c.\ and convex function $f:E\to \R$ satisfies \emph{Jensen's inequality}
\begin{equation}\label{eq:Jensen}
\int_E f\,dp \geq f\bigl(b(p)\bigr).
\end{equation}
Note that the integral on the left hand side is well defined; see \cite[Lemma 2.3.7]{Bacak}.
If $p\in \mathcal{P}_2(E)$, then $b(p)$ admits the more classical characterization of being the unique minimizer of
\[
y \mapsto \int_E d^2(y,x)\,p(dx).
\]
The barycenter map satisfies the following important contraction property (see \cite[Theorem 6.3]{SturmNPC}):
\begin{equation}\label{eq:W1contraction}
d\bigl(b(p),b(q)\bigr) \leq W_1(p,q),\qquad \forall p,q \in \mathcal{P}_1(E).
\end{equation}

For $\kappa>0$ there is also a local analog, due to Kuwae \cite{KuwaeJensen} and Yokota \cite[Theorems B, 25]{YokotaCAT1}.

\begin{thm}\label{thm:barycenter}
Let $(E,d)$ be a complete $\CAT(\kappa)$ space for $\kappa > 0$.
Let $p \in \mathcal{P}(E)$ and suppose that there exist $x_0 \in E$ and $\ep \in (0,D_\kappa/2)$ such that $\operatorname{supp} p \subset \overline{B}(x_0,\ep)$.
Then, the function $y\mapsto \int_E d^2(y,x)\,p(dx)$ admits a unique minimizer, denoted by $b(p) \in E$, which belongs to $\overline{B}(x_0,\ep)$ and is called the barycenter of $p$.
Moreover, for any l.s.c.\ convex function $f:\overline{B}(x_0,\ep) \to \R$, Jensen's inequality \eqref{eq:Jensen} is satisfied.
\end{thm}

\subsubsection{Convex means.}\label{sec:convex-means}
Slightly extending the definition of \cite{SturmNPC}, if $(E,d)$ is a geodesic space and $p \in \mathcal{P}_1(E)$, we will say that a point $x \in E$ is a \emph{convex mean} of $p$ if $\delta_x \lc p$, that is,
\[
f(x) \leq \int_E f\,dp
\]
for any l.s.c.\ convex function $f:E\to \R\cup \{+\infty\}$ such that $\int_E [f]_-\,dp<+\infty$.
If $(E,d)$ is a $\CAT(\kappa)$ space with $\kappa\geq0$, then the class of convex test functions can be reduced, under the assumptions of Lemma \ref{lem:tech}.
The set of all convex means of $p$ will be denoted by $C(p)$ in the sequel.
Whenever the barycenter $b(p)$ is available, Jensen's inequality (together with Lemma \ref{lem:tech} to allow functions with values in $\R\cup\{+\infty\}$) gives $b(p)\in C(p)$.
In particular, we know that $C(p)$ is nonempty for $p\in\mathcal{P}_1(E)$ in the $\kappa=0$ case or in the setting of Theorem \ref{thm:barycenter} in the $\kappa>0$ case.
The set $C(p)$ is always a closed convex subset of $E$.
In the $\kappa=0$ case, $C(p)$ is bounded for every $p\in \mathcal{P}_1(E)$, since for any $x\in C(p)$, it holds
\[
d(x_0,x) \leq \int_E d(x_0,y)\,p(dy),
\]
where $x_0$ is an arbitrary point in $E$, by the convexity of the function $d(x_0,\,\cdot\,)$.
In the $\kappa>0$ case, $C(p)$ is also obviously bounded, whenever $p$ satisfies the assumptions of Theorem \ref{thm:barycenter}.
If $D \subset E$ is a closed convex set such that $p(D)= 1$, then $C(p) \subset D$.

\begin{lem}[Stability along geodesics]\label{rem:convex_mean_geodesic}
Let $p_0,p_1$ be two probability measures on $E$.
Consider a geodesic $(x_t)_{0 \leq t \leq 1}$.
If $x_0$ is a convex mean of $p_0$ and $x_1$ a convex mean of $p_1$, then $x_t$ is a convex mean of $(1-t)p_0 + tp_1$.
\end{lem}

\begin{proof}
Let $f$ be a l.s.c.\ convex function on $E$.
By the convexity of $f$ and the assumptions on $x_0$ and $x_1$,
\[
f(x_t)\leq (1-t)f(x_0)+t f(x_1)
\leq (1-t)\int_E f\,dp_0+t\int_E f\,dp_1
=\int_E f\,d\bigl((1-t)p_0+tp_1\bigr).
\qedhere
\]
\end{proof}

\section{A weak optimal transport formulation}\label{sec:wot}
\subsection{Definition and main result}

Let $(E,d)$ be a complete separable $\CAT(\kappa)$ space with $\kappa \geq 0$.
We define the weak cost function $c : E \times \mathcal{P}_2(E) \to [0,+\infty]$ by
\[
c(x,p) := d^2\bigl(x, C(p)\bigr),
\]
where $d(x,B) := \inf_{y\in B} d(x,y)$ denotes the usual distance from a point to a set and $C(p)$ is the set of convex means of $p$ introduced in Section \ref{sec:convex-means}.
If $C(p)$ is empty, we set $d(x,C(p))=+\infty$.
This section is devoted to studying the weak optimal transport problem
\[
\mathcal{T}_c(\mu,\nu) := \inf_{p:\mu p = \nu} \int_E c(x,p_x)\,\mu(dx),
\]
where the infimum runs over the set of probability kernels $p=(p_x)_{x\in E}$ such that $\int_E p_x(dy)\,\mu(dx) = \nu(dy)$.
This class of transport problems was introduced in \cite{GRST17} (see also \cite{ABC19}); we refer to \cite{BP-survey} for a survey and to \cite{GRST17,ABC19,BBP19,Beiglbock-Fundamental} for theoretical results (existence of optimizers, Kantorovich-type duality, cyclical monotonicity) available in this framework.

In the $\kappa>0$ case, we will always assume for $\mu,\nu \in \mathcal{P}_2(E)$ that the support of $\nu$ is contained in a closed convex set $D$ such that
\begin{equation}\label{eq:kappa>0}
\varepsilon_{\mu,\nu}:=\sup_{x \in \operatorname{supp} \mu} \sup_{y \in D} d(x,y) < \frac{D_\kappa}{2}.
\end{equation}
To unify the cases, when $\kappa = 0$, we will always use the convention $D=E$.

The goal of this section is to prove the following theorem, which contains Theorem \ref{thm:projection_WOT} in the case $\kappa = 0$.

\begin{thm}\label{thm:main-WOT}
Let $(E,d)$ be a complete separable $\CAT(\kappa)$ space with $\kappa \geq 0$, and let $\mu,\nu \in \mathcal{P}_2(E)$.
In the $\kappa>0$ case, we assume \eqref{eq:kappa>0}. Then
\[
\min_{\eta \lc \nu} W_2^2(\mu,\eta) \;=\; \min_{p:\mu p = \nu} \int_E d^2\bigl(x,C(p_x)\bigr)\,\mu(dx) \;=\; \max_{f:\text{convex}} \left\{ \int_E Qf\,d\mu - \int_E f\,d\nu \right\},
\]
where the maximum is taken over all lower semicontinuous, $\nu$-integrable, convex functions $f:E \to \R \cup \{+\infty\}$ and $Qf(x) := \inf_{y\in D}\, \{f(y)+d^2(x,y)\}$.
The solution $\bar \mu \lc \nu$ of the former minimizing problem is unique.
Moreover, solutions $\bar{\mu},p,f$ of these three problems are related by a map $T$ defined $\mu$-almost everywhere as follows:
\begin{enumerate}
\item The plan $(\mathrm{id},T)_\#\mu$ is the unique optimal transport plan from $\mu$ to $\bar{\mu}$;
\item For $\mu$-almost all $x$, $T(x)$ is the metric projection of $x$ to $C(p_x)$;
\item $T$ is the proximal operator associated with the function $f$, i.e., for $\mu$-almost all $x$,
\[ T(x) = \argmin_{y \in D} \Bigl\{f(y) + d^2(x,y)\Bigr\}.
\]
\end{enumerate}
\end{thm}

Note that our definition of proximal operator differs from the usual one by a factor $1/2$ (in other words, $T$ is the operator at time $1/2$; cf.\ \cite{Bacak}).
Note also that in the case of the flat Euclidean space $E = \mathbb{R}^d$, the set of convex means of a measure $p$ only contains its barycenter, and the cost $c$ is simply the barycentric cost $|x - \int_{\R^d} y \,p(dy)|^2$; hence Theorem \ref{thm:main-WOT} generalizes \cite[Theorem 1.2]{GJ20}.
See also \cite{Kim-Ruan,pramenkovic2025} for similar projection problems involving general stochastic orders.

In the first part of the section, we study in detail the cost function $c : (x,p) \mapsto d^2(x,C(p))$: we show that it is continuous with respect to the product topology on $E\times \mathcal{P}_2(E)$ (with the appropriate localization when $\kappa>0$). The second part of the section is devoted to the proof of Theorem \ref{thm:main-WOT}.

\subsection{Study of the cost function}
Recall that $c$ denotes the cost function defined in the preceding subsection.

\begin{prop}\label{lem:lsc}
Let $(E,d)$ be a complete separable $\CAT(\kappa)$ space with $\kappa \geq 0$.
\begin{enumerate}
\item If $\kappa = 0$, then $c$ is continuous on $E\times \mathcal{P}_2(E)$ for the product topology, and is convex in the second variable with respect to convex interpolations.
\item When $\kappa > 0$, let $D$ be a closed convex set included in a closed ball of radius $\varepsilon <D_\kappa/2$ and set $\widetilde D:=\{x \in E : \sup_{y \in D} d(x,y) \leq \ep \}$.
Then $c$ is finite valued, continuous for the product topology, and  convex in the second variable on $\widetilde{D}\times \mathcal{P}_2(D)$ with respect to convex interpolations.
\end{enumerate}
\end{prop}
Note that, under \eqref{eq:kappa>0}, $D \subset \overline{B}(x_0,\ep_{\mu,\nu})$ holds for any $x_0 \in \operatorname{supp}\mu$, and $\operatorname{supp}\mu \subset \widetilde{D}$ with $\ep=\ep_{\mu,\nu}$.
We first record a variant of Cantor's intersection theorem for complete $\CAT(\kappa)$ spaces which will be used in the proof; see e.g.\ \cite[Proposition 2.1.16]{Bacak} and \cite[Lemma 5.1]{CL10}.

\begin{lem}[Finite-intersection property]\label{lem:fip}
Let $(E,d)$ be a complete $\CAT(\kappa)$ space with $\kappa \geq 0$.
Let $(K_i)_{i\in I}$ be a family of nonempty closed convex subsets of $E$, such that every finite subfamily has nonempty intersection and one member of the family is included in a ball of radius $< D_\kappa/2$.
Then $\bigcap_{i\in I} K_i \neq \varnothing$.
\end{lem}

For nonempty bounded closed sets $A,B\subset E$, let
\[
d_H(A,B) := \max\left\{ \sup_{x\in A} d(x,B),\; \sup_{y\in B} d(y,A) \right\}
\]
denote the Hausdorff distance between $A$ and $B$.
The proof of the continuity of $c$ is based on the Lipschitz stability of the set of convex means with respect to the 1-Wasserstein distance.

\begin{prop}\label{prop:stability}
Let $(E,d)$ be a complete separable $\CAT(\kappa)$ space with $\kappa \geq 0$.
Let $p,q \in \mathcal{P}_1(E)$ be two probability measures supported in a convex closed subset $D \subset E$.
When $\kappa >0$, suppose that $D$ is included in a closed ball $B$ of radius $\varepsilon <D_\kappa/2$.
Then the following hold.
\begin{enumerate}
\item For any coupling $\pi \in \Pi(p,q)$, we have
\[ C(p) = \operatorname{pr}_1\bigl(C(\pi)\bigr), \qquad C(q) = \operatorname{pr}_2\bigl(C(\pi)\bigr),\]
where $\operatorname{pr}_1,\operatorname{pr}_2$ are the coordinate projections from $E \times E$ to $E$.
\item As a consequence, we have
\[
\begin{array}{ll}
d_H\bigl(C(p), C(q)\bigr) \leq W_1(p,q) \quad &\text{if } \kappa =0, \\[3pt]
d_H\bigl(C(p), C(q)\bigr) \leq  C W_1^{\alpha}(p,q) \quad &\text{if } \kappa > 0,
\end{array}
\]
with constants $C>0$ and $\alpha \in (0,1)$ depending on $\kappa$ and $\ep$.
\end{enumerate}
\end{prop}

Note that (2) above extends in particular the Wasserstein contraction property of barycenters \eqref{eq:W1contraction} to the set-valued context of convex means.

\proof[Proof of Proposition \ref{prop:stability}.]

\noindent\emph{Proof of item 1.}
By symmetry, it is enough to prove the statement for the first marginal, namely $C(p)=\operatorname{pr}_1(C(\pi))$.
The inclusion $\operatorname{pr}_1(C(\pi))\subset C(p)$ is immediate: if $(x,y)\in C(\pi)$ and $f:E \to \R \cup \{+\infty\}$ is a l.s.c.\ convex function, then $f\circ\operatorname{pr}_1:E\times E \to \R \cup \{+\infty\}$ is also l.s.c.\ convex and thus
\[
f(x)\leq \int_{E\times E} f(u)\,\pi(dudv)=\int_E f\,dp,
\]
provided that the integral makes sense.

Let us prove the converse inclusion.
In the $\kappa=0$ case, we set $D=B=E$ in order to unify the reasoning.
Fix $x_0 \in C(p)$ and an arbitrary point $o\in E$.
Let $\Phi_0$ be the function defined by $\Phi_0(x,y) := d(o,y)$ in the case $\kappa=0$, and $\Phi_0:=0$ if $\kappa>0$.
Let $\Phi_1,\ldots,\Phi_m$ be convex, uniformly continuous, $\pi$-integrable functions bounded from below on $B\times B$.
Put
\[
c_i:=\int_{B\times B}\Phi_i\,d\pi,\qquad 0\leq i\leq m.
\]
Note that $c_0<\infty$ since $q \in \mathcal{P}_1(E)$.
Define
\[
L:=\Bigl\{r\in\R^{m+1}: \exists\, y\in D\text{ such that }\Phi_i(x_0,y)\leq r_i,\, 0\leq i\leq m\Bigr\}.
\]
Then $L$ is convex and satisfies $L+\mathbb{R}^{m+1}_+=L$.
Moreover, $L$ is closed.
Indeed, let $r^n\in L$ be such that $r^n\to r$, and $y_n$ witnesses $r^n$.
Let us show that $r\in L$.
First, $(y_n)$ is bounded.
This is obvious in the $\kappa>0$ case, since $D$ is bounded.
In the $\kappa=0$ case, it holds $d(o,y_n)=\Phi_0(x_0,y_n)\leq r_0^n$, which implies the boundedness of $(y_n)$.
By the finite-intersection property (Lemma \ref{lem:fip}) applied to the decreasing closed convex hulls $\overline{\operatorname{co}}\{y_n:n\geq k\}$, there exists $y$ in their intersection.
For each $k$, $y\in \overline{\operatorname{co}}\{y_n:n\geq k\}$, thus by the convexity and continuity of each $\Phi_i(x_0,\,\cdot\,)$,
\[
\Phi_i(x_0,y)\leq \sup_{n\geq k}\Phi_i(x_0,y_n).
\]
Hence, by letting $k\to \infty$,
\[
\Phi_i(x_0,y)\leq \limsup_{n\to \infty}\Phi_i(x_0,y_n) \leq r_i,
\]
and so $r\in L$.

We claim that $c=(c_0,\dots,c_m)\in L$.
If not, since $L$ is a closed convex upper set, the hyperplane separation theorem between $L$ and $\{c\}$ gives some nonzero $a\in\mathbb{R}^{m+1}_+$ such that
\[
\inf_{y\in D}\sum_{i=0}^ma_i\Phi_i(x_0,y) > \sum_{i=0}^m a_i c_i.
\]
Set $g(x):=\inf_{y\in D}\sum_{i=0}^m a_i\Phi_i(x,y)$, $x\in B$, which is clearly bounded from below.
As the infimum in one variable of a jointly convex uniformly continuous function, the function $g$ is convex and uniformly continuous on $
B$.
Since $x_0 \in C(p)$,
\[
g(x_0)\leq \int_B g\,dp = \int_{B\times B}g(x)\,\pi(dxdy) \leq \sum_{i=0}^m a_i c_i,
\]
contradicting the above separation inequality.
Thus $c\in L$.

Therefore, for every finite family $\Phi_1,\ldots,\Phi_m$, there exists $y\in D$ such that
\[
\Phi_i(x_0,y)\leq \int_{B\times B}\Phi_i\,d\pi,\qquad 0\leq i\leq m.
\]
Equivalently, the closed convex sets
\[
K_\Phi:=\Bigl\{y\in D:\Phi(x_0,y)\leq \int_{B\times B}\Phi\,d\pi\Bigr\},
\]
indexed by convex, uniformly continuous, $\pi$-integrable and lower bounded functions $\Phi$ on $B\times B$, have the finite-intersection property, and $K_{\Phi_0}$ is bounded.
By Lemma \ref{lem:fip}, there exists $y_0\in\bigcap_\Phi K_\Phi$, which, according to Lemma \ref{lem:tech}, exactly means $(x_0,y_0)\in C(\pi)$.
Thus $C(p)\subset \operatorname{pr}_1(C(\pi))$.

\noindent\emph{Proof of item 2; case $\kappa=0$.}
Let $\pi\in\Pi(p,q)$ be an optimal $W_1$-coupling, and fix $x \in C(p)$.
By item 1, there exists $y \in E$ such that $(x,y)\in C(\pi)$.
Applying the defining inequality for $C(\pi)$ to the convex function $\Phi(u,v):=d(u,v)$ gives
\[
d(x,y)\leq \int_{E\times E} d(u,v)\,\pi(dudv) =W_1(p,q).
\]
Since $y \in C(q)$ again by item 1, we get $d(x,C(q))\leq W_1(p,q)$.
Taking the supremum over $x \in C(p)$, and arguing symmetrically on the second coordinate, we obtain
\[
d_H\bigl(C(p),C(q)\bigr) \leq W_1(p,q).
\]

\noindent\emph{Proof of item 2; case $\kappa>0$.}
Let $\Phi:B\times B \to \R_+$ be the uniformly continuous jointly convex function used in the proof of Lemma \ref{lem:tech}.
Reasoning as in the $\kappa=0$ case yields
\[
\max\left(\sup_{x\in C(p)}\inf_{y\in C(q)}\Phi(x,y),\, \sup_{y\in C(q)}\inf_{x\in C(p)}\Phi(x,y)\right) \leq \inf_{\pi \in \Pi(p,q)}\int_{D\times D} \Phi(u,v)\,\pi(dudv).
\]
Then, using \eqref{eq:ineq-Yokota} and the boundedness of $D$, we conclude that
\[
d_H\bigl(C(p),C(q)\bigr) \leq  CW_1^{\alpha}(p,q)
\]
for $C=(C_0/c_0)^{1/m} \operatorname{diam}(D)^{(m-1)/m}$ and $\alpha=1/m$ (for $m>1$ in \eqref{eq:ineq-Yokota}).
\endproof

\proof[Proof of Proposition \ref{lem:lsc}]
We give the proof in the $\kappa>0$ case, the other case is simpler.
By the triangle inequality, for all $x,y\in \widetilde{D}$ and $p,q\in\mathcal{P}_2(D)$,
\[ \bigl|d\bigl(x,C(p)\bigr)-d\bigl(y,C(q)\bigr)\bigr| \leq d(x,y)+d_H\bigl(C(p),C(q)\bigr).
\]
Together with the H\"older estimate of Proposition \ref{prop:stability}, it shows that the cost function $c$ is continuous for the product topology with the $1$-Wasserstein topology and hence with the $2$-Wasserstein topology.
Now, let us prove the convexity with respect to the second variable.
Fix $x \in \widetilde{D}$.
By Lemma \ref{lem:Ohta}, the function $d^2(x,\,\cdot\,)$ is convex on the closed ball of center $x$ and radius $\varepsilon$, and thus on $D$ (by the definition of $\widetilde{D}$).
Let
$p_0,p_1\in\mathcal P_2(D)$ and put $p_t:=(1-t)p_0+tp_1 \in \mathcal{P}_2(D)$, $t\in [0,1]$.
According to Lemma \ref{rem:convex_mean_geodesic}, if
$(y_t)_{t\in [0,1]}$ is a geodesic connecting $y_0 \in C(p_0)$ to $y_1 \in C(p_1)$, then $y_t$ belongs to $C(p_t)$.
Hence
\[
 c(x,p_t)\leq  d^2(x,y_t)\leq (1-t)d^2(x,y_0)+td^2(x,y_1).
\]
Taking the infimum over $y_0,y_1$ then yields
\[
 c(x,p_t) \leq  (1-t)c(x,p_0)+t c(x,p_1),
\]
which completes the proof.
\endproof

\subsection{Convex and lower semicontinuous envelopes}\label{sec:envelope}

Before turning to the proof of Theorem \ref{thm:main-WOT}, we need to introduce the notions of convex and l.s.c.\ envelopes of a function.
In this subsection, $D$ is a convex subset of a geodesic space $E$.

Let $f:D \to \R \cup\{+\infty\}$ be a measurable function satisfying the lower bound $f \geq -a(1 + d^2(x_0,\,\cdot\,))$ for some $a\geq0$ and $x_0\in E$.
The \emph{convex envelope} of $f$ is the function $\overline{f}:D\to\R \cup\{\pm \infty\}$ defined by
\[
\overline f(x) := \inf\biggl\{ \int_D f\,dp \;:\; p \in \mathcal{P}_2(D),\ x\in C(p) \biggr\}.
\]
Since $x\in C(\delta_x)$, the infimum is taken over a nonempty set, and $\overline f(x)\leq f(x)$ for every $x\in D$.

\begin{lem}\label{lem:cvx-envelope}
Let $f:D \to \R \cup \{+\infty\}$ be such that $\overline f$ does not take the value $- \infty$.
Then the convex envelope $\overline f$ is convex.
Moreover, if $f$ is l.s.c.\ and convex, then $\overline{f}=f$.
\end{lem}

\proof
Let $x_0,x_1\in D$ and $(x_t)_{t\in[0,1]}$ be a geodesic in $D$, and fix $p_0,p_1\in\mathcal{P}_2(D)$ such that $x_i\in C(p_i)$ for $i=0,1$.
By Lemma \ref{rem:convex_mean_geodesic}, $x_t\in C((1-t)p_0+tp_1)$, thereby
\[
\overline f(x_t)
\leq \int_D f\,d\bigl((1-t)p_0+tp_1\bigr)
=(1-t)\int_D f\,dp_0 +t\int_D f\,dp_1.
\]
Taking the infimum over all admissible $p_0,p_1$ gives
\[
\overline f(x_t)\leq (1-t)\overline f(x_0)+t\overline f(x_1).
\]

Now assume that $f:D \to \R \cup \{+\infty\}$ is l.s.c.\ and convex.
We already know that $\overline f \leq f$.
To see the converse, let $x \in D$.
Since $f$ is l.s.c.\ and convex, by the definition of convex mean, for all $p$ with $x\in C(p)$ it holds
\[
f(x)\leq \int_D f\,dp.
\]
Taking the infimum over $p$ in the right hand side yields $f(x)\leq \overline f(x)$.
\endproof

Next, given a function $f:D\to \R\cup\{+\infty\}$, its \emph{lower semicontinuous envelope} $\widetilde{f}$ is defined as the greatest l.s.c.\ function below $f$.
It is in fact given by
\begin{equation}\label{eq:lsc}
\widetilde{f}(x) = \liminf_{y\to x} f(y) = \sup_{r>0} \inf_{y \in B(x,r)\cap D} f(y),\qquad \forall x\in D.
\end{equation}
We will use the following result showing that the l.s.c.\ envelope of a convex function is still convex.

\begin{lem}\label{lem:lsc-envelope}
Let $(E,d)$ be a complete $\CAT(\kappa)$ space with $\kappa \geq 0$ and $D \subset E$ be a convex closed subset.
When $\kappa >0$, suppose that $D$ is included in a ball of radius $\ep <D_\kappa/2$.
If $f:D \to \R\cup\{+\infty\}$ is a convex function such that $\widetilde{f}$ does not take the value $-\infty$, then $\widetilde{f} : D\to \R \cup\{+\infty\}$ is convex as well.
\end{lem}

\proof
Fix $x_0,x_1\in D$ and denote by $\gamma$ the unique constant speed geodesic connecting $x_0$ to $x_1.$
Fix $t\in (0,1)$ and let us prove that
\begin{equation}\label{eq:tildefcvx}
\widetilde{f}\bigl(\gamma(t)\bigr)\leq (1-t) \widetilde{f}(x_0) + t \widetilde{f}(x_1).
\end{equation}
If $\widetilde f(x_0)=+\infty$ or $\widetilde f(x_1)=+\infty$, there is nothing to prove.
So let us assume that both endpoint values are finite.
By the sequential characterization \eqref{eq:lsc} of the l.s.c.\ envelope, for each $i\in\{0,1\}$ there is a sequence $(x_i^n)_n$ in $D$ such that
\[
x_i^n\longrightarrow x_i
\qquad\text{and}\qquad
f(x_i^n)\longrightarrow \widetilde f(x_i).
\]
Let $\gamma^n:[0,1] \to D$ be the geodesic connecting $x_0^n$ to $x_1^n$.
According to \cite[Proposition II.1.4(1)]{BH99}, geodesic segments in a $\CAT(\kappa)$ space joining endpoints at distance less than $D_\kappa$ are unique and depend (uniform) continuously on their endpoints.
Hence $\gamma^n(t) \to \gamma(t)$ and, by \eqref{eq:lsc},
\[
\widetilde f\bigl(\gamma(t)\bigr)\leq \liminf_{n\to\infty} f\bigl(\gamma^n(t)\bigr).
\]
On the other hand, the convexity of $f$ gives
\[
f\bigl(\gamma^n(t)\bigr)
\leq (1-t)f(x_0^n)+t f(x_1^n).
\]
Taking the limits as $n \to \infty$ and using the convergence of the endpoint values, we obtain \eqref{eq:tildefcvx}.
\endproof

We will also need the following lemma showing that the infimum convolution operator $Q$ (recall Theorem \ref{thm:main-WOT}) does not see the difference between a function and its l.s.c.\ envelope.

\begin{lem}\label{lem:infconv-lsc}
For any $f:D\to \R\cup\{+\infty\}$, it holds $Qf=Q\widetilde{f}$.
\end{lem}

\proof
By construction $\widetilde{f} \leq f$, therefore
\[
Q\widetilde{f}(x) = \inf_{y\in D}\{\widetilde{f}(y)+d^2(x,y)\} \leq \inf_{y\in D}\{f(y)+d^2(x,y)\} = Qf(x).
\]
On the other hand, for all $x \in E$, it holds
\[
Qf(x) \leq f(z)+d^2(x,z)
\]
for all $z\in D$.
Therefore, taking the $\liminf$ as $z \to y$ gives
\[
Qf(x) \leq \widetilde{f}(y)+d^2(x,y).
\]
Taking the infimum in $y\in D$ yields $Qf(x)\leq Q\widetilde{f}(x)$ and completes the proof.
\endproof

\subsection{Proof of the main result}

\proof[Proof of Theorem \ref{thm:main-WOT}]
In the sequel, $D$ denotes the closed convex set satisfying \eqref{eq:kappa>0} and let $\ep=\ep_{\mu,\nu}$.
Set $\widetilde{D} := \{x \in E : \sup_{y\in D} d(x,y)\leq \ep\}$ with the convention $D=\widetilde{D}=E$ in the case $\kappa=0$.
Recall that, under \eqref{eq:kappa>0}, we have $\mu \in \mathcal{P}_2(\widetilde D)$ and $\nu \in \mathcal{P}_2(D)$.

By Proposition \ref{lem:lsc}, the cost function $c:\widetilde{D} \times \mathcal{P}_2(D) \to \R_+$ is continuous for the product topology and convex in its second variable.
In particular, the condition $(\mathrm{C})$ of \cite{Beiglbock-Fundamental} is satisfied.
Moreover, for all $x\in \widetilde{D}$ and $p \in \mathcal{P}_2(D)$, the triangle inequality and the convexity of $d^2(x_0,\,\cdot\,)$ yield
\[
c(x,p) \leq 2d^2(x,x_0) +2d^2\bigl(x_0,C(p)\bigr)
\leq 2d^2(x,x_0)+2 \int_D d^2(x_0,y)\,p(dy),
\]
for any $x_0 \in \widetilde{D}$.
Thus, the condition $(\mathrm{B})$ of \cite{Beiglbock-Fundamental} is also satisfied.
Therefore, the fundamental theorem of weak optimal transport \cite[Theorem 1.2]{Beiglbock-Fundamental} (see also \cite[Theorems 1.2 and 1.3]{BBP19}) ensures that
\begin{equation}\label{eq:duality-WOT}
\mathcal{T}_c(\mu,\nu) = \min_{p:\mu p = \nu} \left\{ \int_{\widetilde D} c(x,p_x) \,\mu(dx) \right\} \;=\; \max_{f} \left\{ \int_{\widetilde D} Q_c f\,d\mu - \int_D f\,d\nu \right\},
\end{equation}
where the maximum is taken over $\nu$-integrable functions $f:D\to\R \cup \{+\infty\}$, and
\[
Q_c f(x) := \inf_{\rho \in \mathcal{P}_2(D)} \biggl\{ \int_D f\,d\rho + c(x,\rho) \biggr\}.
\]
Note that $\int_{\widetilde D} Q_cf\,d\mu$ always makes sense in $\R \cup \{-\infty\}$, and that the maximum can be further restricted to those $f$ such that $\int_{\widetilde D} Q_cf\,d\mu >-\infty$.
Let $f$ be such a function.
Then, for all $y \in D$, taking $\rho=\delta_y$ in the definition of $Q_cf(x_0)$ implies
\[
f(y) \geq Q_cf(x_0)- d^2(x_0,y),
\]
where $x_0$ is any point in $\widetilde{D}$ such that $Q_cf(x_0) >-\infty$.
Let $\overline f$ be the convex envelope introduced in Section \ref{sec:envelope}.
Writing $c(x,\rho)$ as $\inf_{y\in C(\rho)} d^2(x,y)$ and exchanging the two infima, we get
\begin{align*}
Q_c f(x) &= \inf_{\rho \in \mathcal{P}_2(D)}\, \inf_{y\in C(\rho)} \biggl\{ \int_D f\,d\rho + d^2(x,y) \biggr\}
= \inf_{y\in D}\, \inf_{\rho:y\in C(\rho)} \biggl\{ \int_D f\,d\rho + d^2(x,y) \biggr\} \\
&= \inf_{y\in D} \bigl\{ \overline f(y) + d^2(x,y) \bigr\},
\end{align*}
that is, $Q_c f = Q\overline f$.
This shows also that $\overline f$ does not take the value $-\infty$, otherwise $Q_c f$ would be identically $-\infty$.
Denoting by $\widetilde{\overline f}$ the l.s.c.\ envelope of the convex envelope $\overline{f}$, we deduce from Lemma \ref{lem:infconv-lsc} that
\[
 Q_cf=Q\overline f=Q\widetilde{\overline f}.
\]
Consequently,
\[
 \widetilde{\overline f}(y)
 \ge Q_cf(x_0)-d^2(x_0,y),\qquad \forall y\in D,
\]
so the negative part of $\widetilde{\overline f}$ is finite valued and $\nu$-integrable.
Moreover, since $\widetilde{\overline f}\le \overline{f} \le f$, the positive part of $\widetilde{\overline f}$ is also $\nu$-integrable.
Therefore $\widetilde{\overline f}\in L^1(\nu)$ and one gets
\begin{equation}\label{eq:envelope-ineq}
 \int_{\widetilde D} Q_cf\,d\mu-\int_D f\,d\nu
 \le
 \int_{\widetilde D} Q\widetilde{\overline f}\,d\mu
 -\int_D \widetilde{\overline f}\,d\nu.
\end{equation}
Moreover, according to Lemma \ref {lem:lsc-envelope}, $\widetilde{\overline{f}}$ is convex and l.s.c.
Conversely, if $g:D \to \R\cup \{+\infty\}$ is a $\nu$-integrable, convex and l.s.c.\ function, then Lemma \ref{lem:cvx-envelope} implies $\overline{g}=g$, so that $\widetilde{\overline{g}} = g$.
Together with \eqref{eq:envelope-ineq}, it means that the maximum in \eqref{eq:duality-WOT} can be restricted to this class of functions:
\begin{equation}\label{eq:TcQ}
\mathcal{T}_c(\mu,\nu) \;=\; \max_{f} \left\{ \int_{\widetilde D} Qf\,d\mu - \int_D f\,d\nu \right\},
\end{equation}
where the maximum is taken over all $\nu$-integrable, convex and l.s.c.\ functions $f:D\to\R \cup \{+\infty\}$.

Let $f$ be a $\nu$-integrable, convex and l.s.c.\ function on $D$, and let $\eta \lc \nu$.
The convex order gives $\int_D f\,d\eta \leq \int_D f\,d\nu$.
Hence
\[
\int_{\widetilde D} Qf\,d\mu - \int_D f\,d\nu \;\leq\; \int_{\widetilde D} Qf\,d\mu - \int_D f\,d\eta \;\leq\; W_2^2(\mu,\eta),
\]
where the latter inequality follows from the definition of $Qf$.
Taking the supremum over $f$ on the left and the infimum over $\eta$ on the right, we obtain
\begin{equation}\label{eq:WOT-geq-proj}
    \mathcal{T}_c(\mu,\nu) \;\leq\; \inf_{\eta \lc \nu} W_2^2(\mu,\eta).
\end{equation}

Let $p$ be an optimal kernel for $\mathcal{T}_c(\mu,\nu)$. For $\mu$-almost every $x$, the set $C(p_x)$ is a nonempty closed convex subset of $D$, and the metric projection $T(x)$ of $x$ to $C(p_x)$ is well defined (see e.g.\ \cite[Proposition 2.6]{SturmNPC} for $\kappa=0$ and \cite[Lemma 2.8]{KuwaeJensen} for $\kappa>0$).
Using Proposition \ref{prop:stability} together with Lemma \ref{lem:joint-projection} in Appendix \ref{sec:measurablility}, it is easy to see that $T$ is measurable, details are left to the reader.
Set $\bar\mu := T_\#\mu$.
Then
\begin{equation*}
\mathcal{T}_c(\mu,\nu) = \int_{\widetilde D} d^2\bigl(x,T(x)\bigr)\,\mu(dx) \;\geq\; W_2^2(\mu,\bar\mu).
\end{equation*}
Moreover, for every convex l.s.c.\ function $g$, since $T(x)\in C(p_x)$,
\[
\int_D g\,d\bar\mu = \int_{\widetilde D} g\bigl(T(x)\bigr)\,\mu(dx) \leq \int_{\widetilde D} \int_D g(y)\,p_x(dy)\,\mu(dx) = \int_D g\,d\nu,
\]
so that $\bar\mu \lc \nu$, and therefore
\[
\mathcal{T}_c(\mu,\nu) \;\geq\; W_2^2(\mu,\bar\mu) \;\geq\; \inf_{\eta \lc \nu} W_2^2(\mu,\eta).
\]
The opposite inequality \eqref{eq:WOT-geq-proj} shows that equality holds and that $\bar \mu$ is a minimizer for the right hand side, thereby
\[ \mathcal{T}_c(\mu,\nu) \;=\; \min_{\eta \lc \nu} W_2^2(\mu,\eta).
\]
This yields the first assertion of Theorem \ref{thm:main-WOT}.
As for the optimality conditions, note that (1), (2) of Theorem \ref{thm:main-WOT} describe how we constructed $T$.
Now it remains to prove the uniqueness of $T$ and that $T$ is the proximal operator for $f$ as in (3).

\emph{Uniqueness of $\bar \mu$ and of the map $T$.}
Put $\eta_0:=\bar\mu$ and take $\eta_1 \lc \nu$ such that
\[
W_2^2(\mu,\eta_1) = \mathcal{T}_c(\mu,\nu).
\]
It follows from the convex order that $\eta_0(D)=\eta_1(D)=1$.
Let $(Z,X_0, X_1)$ be random variables constructed on the same probability space such that $(Z,X_0)$ and $(Z,X_1)$ are optimal couplings for $(\mu,\eta_0)$ and $(\mu,\eta_1)$, respectively (such random variables are easily obtained from any optimal couplings by a gluing argument, we omit details).
We use the generalized geodesic between $\eta_0$ and $\eta_1$ with base $\mu$: let $X_{1/2}:=[X_0,X_1]_{1/2}$ be the midpoint of $X_0$ and $X_1$, and let $\eta_{1/2}$ denote its law.
For any l.s.c.\ convex function $g:D\to\mathbb{R}$, convexity gives $g(X_{1/2}) \leq \frac12 g(X_0)+\frac12 g(X_1)$, so that
\[
\int_D g\,d\eta_{1/2} \leq \frac12 \int_D g\,d\eta_0 + \frac12 \int_D g\,d\eta_1 \leq \int_D g\,d\nu.
\]
Hence $\eta_{1/2}\lc \nu$, which ensures
\[
\mathbb{E}\, d^2(Z,X_{1/2}) \;\geq\; W_2^2(\mu,\eta_{1/2}) \;\geq\; \mathcal{T}_c(\mu,\nu).
\]
On the other hand, the function $d^2(z,\,\cdot\,)$ is $k$-convex on $D$ for every $z$ in the support of $\mu$, with $k=2$ if $\kappa=0$, and, by Lemma \ref{lem:Ohta} together with condition \eqref{eq:kappa>0}, with $k=2\sqrt{\kappa}\,\ep/\tan(\sqrt{\kappa}\,\ep)$ if $\kappa>0$.
Hence
\[
d^2(Z,X_{1/2}) \leq \frac12 d^2(Z,X_0) + \frac12 d^2(Z,X_1) - \frac{k}{8}\, d^2(X_0,X_1).
\]
Taking expectations and combining the two previous inequalities yields $\mathbb{E}\, d^2(X_0,X_1) \leq 0$, thus $X_0=X_1$ almost surely.
Therefore $\eta_1= \bar{\mu}$ and the optimal coupling between $\mu$ and the projection $\bar{\mu}$ is unique (since the optimal couplings between $\mu$ and $\eta_0,\eta_1$ could be chosen arbitrarily), which is the one given by $T$.

\emph{$T$ is the proximal operator associated with $f$.}
Let $f$ be an optimal convex function in \eqref{eq:TcQ}.
Since $\bar\mu \lc \nu$, we have $\int_D f\,d\bar\mu \leq \int_D f\,d\nu$, and, for all $x \in \widetilde{D}$, we have $Qf(x) \leq f(T(x)) + d^2(x,T(x))$ by the definition of $Qf$.
Therefore
\begin{align*}
\mathcal{T}_c(\mu,\nu) &= \int_{\widetilde D} Qf\,d\mu - \int_D f\,d\nu
\leq \int_{\widetilde D} Qf\,d\mu - \int_D f\,d\bar\mu
= \int_{\widetilde D} \bigl\{ Qf(x) - f\bigl(T(x)\bigr) \bigr\}\,\mu(dx) \\
&\leq \int_{\widetilde D} d^2\bigl(x,T(x)\bigr)\,\mu(dx) = \mathcal{T}_c(\mu,\nu),
\end{align*}
so that all these inequalities become equality.
In particular, for $\mu$-almost every $x$,
\[
Qf(x) = f\bigl(T(x)\bigr) + d^2\bigl(x,T(x)\bigr),
\]
that is, $T(x)$ attains the infimum defining $Qf(x)$.
Moreover, the function $y \mapsto f(y) + d^2(x,y)$ is the sum of a convex function and a $k$-convex function with $k>0$, thus it is $k$-convex and admits at most one minimizer.
Consequently, for $\mu$-almost every $x$, $T(x)$ is the unique minimizer of $y \mapsto f(y)+d^2(x,y)$, which is the desired proximal characterization as in (3).
\endproof

\subsection{Consequences of Theorem \ref{thm:main-WOT}}

The strong duality established in Theorem \ref{thm:main-WOT} --- the existence of an optimal convex function $f$ playing the role of a Lagrange multiplier --- has two useful consequences, recorded below:
a weak cyclical monotonicity property satisfied by optimal kernels, and a version of Strassen's theorem adapted to the convex order.

\begin{cor}\label{cor:cyclical}
In the setting of Theorem \ref{thm:main-WOT}, any kernel $p$ optimal for the problem
\[
\min_{p:\mu p = \nu} \int d^2\bigl(x, C(p_x)\bigr)\, \mu(dx)
\]
is $c$-monotone, with $c(x,p) = d^2(x, C(p))$.
In particular, there exists a subset $A \subset \operatorname{supp}\mu$ with $\mu(A)=1$ such that for all $x,y \in A$ and all $q_1, q_2 \in \mathcal{P}_2(D)$ satisfying
\[
q_1 + q_2 = p_x + p_y,
\]
it holds
\[
c(x,p_x) + c(y,p_y) \leq c(x,q_1) + c(y,q_2).
\]
\end{cor}

\proof
The $c$-monotonicity follows from \cite[Corollary 2.13]{Beiglbock-Fundamental}.
Indeed, as shown in the proof of Theorem \ref{thm:main-WOT}, Assumptions $(\mathrm{B})$ and $(\mathrm{C})$ of \cite{Beiglbock-Fundamental} are granted.
Thus, for any optimal kernel $p$, the corresponding coupling $\pi(dxdy) = p_x(dy)\mu(dx)$ is $c$-monotone, in the sense that there exists $\Gamma \subset \widetilde{D} \times \mathcal{P}_2(D)$ such that $\mu(\{x : (x,p_x) \in \Gamma\})=1$ and, for every
$(x_1,\rho_1),\ldots,(x_n,\rho_n) \in \Gamma$
and all
$\widetilde{\rho}_1,\ldots,\widetilde{\rho}_n \in \mathcal{P}_2(D)$
with
\[
    \sum_{i=1}^{n} \rho_i
    =
    \sum_{i=1}^{n} \widetilde{\rho}_i,
\]
it holds
\[
    \sum_{i=1}^{n} c(x_i,\rho_i)
    \leq
    \sum_{i=1}^{n} c(x_i,\widetilde{\rho}_i).
\]
Letting $A:=\{x \in \operatorname{supp}\mu : (x,p_x) \in \Gamma\}$ and taking $(x_1,\rho_1)=(x,p_x)$ and $(x_2,\rho_2)=(y,p_y)$ for $x,y \in A$ shows the claim.
\endproof

The following result is a version of the well known Strassen's theorem \cite{Strassen} for the convex order.

\begin{thm}[Strassen's theorem for the convex order]\label{thm:strassen}
Let $(E,d)$ be a complete separable $\CAT(\kappa)$ space with $\kappa\geq 0$, and let $\mu,\nu\in\mathcal{P}_2(E)$.
Assume that \eqref{eq:kappa>0} holds in the case $\kappa>0$.
Then $\mu\lc \nu$ if and only if there exists a probability kernel $p=(p_x)_{x\in E}$ such that $\mu p=\nu$ and $x\in C(p_x)$ for $\mu$-almost every $x\in E$.
\end{thm}

\proof
The existence of $p=(p_x)_{x\in E}$ such that $\mu p=\nu$ and $x\in C(p_x)$ for $\mu$-almost every $x\in E$ immediately implies $\mu\lc \nu$.
Conversely, suppose that $\int_D g\,d\mu \leq \int_D g\,d\nu$ for all l.s.c.\ convex functions $g:E\to \R\cup\{+\infty\}$.
Since $Qg\leq g$ on $D$ and $\mu(D)=1$, we also have $\int_D Qg\,d\mu \leq \int_D g\,d\nu$. Therefore, by the duality formula of Theorem \ref{thm:main-WOT}, we get $\mathcal{T}_c(\mu,\nu)=0$.
Then an optimal kernel $p$ satisfies $\int_D d^2(x,C(p_x))\,\mu(dx)=0$, and hence $x \in C(p_x)$ for $\mu$-almost all $x$.
\endproof

\begin{rem}
In the $\kappa>0$ case, the assumption \eqref{eq:kappa>0} of Theorem \ref{thm:strassen} is not optimal.
For instance, assuming $E$ is a locally compact $\mathrm{CAT}(\kappa)$ space with $\kappa>0$, Theorem 2.3 of \cite{Ciosmak2023a} together with item 2. of Proposition \ref{lem:tech}, shows that if $\mu \lc \nu$ and $\nu$ is supported on a closed convex set $D$ included in a ball $\overline{B}(x_0,\varepsilon)$ with $\varepsilon<D_\kappa/2$, then there exists a probability kernel $p=(p_x)_{x\in E}$ such that $\mu p=\nu$ and $x\in C(p_x)$ for $\mu$-almost every $x\in E$. We refer to \cite{Ciosmak2023a,Ciosmak2023,Ciosmak-Cor} and \cite{pramenkovic2025} for other variants of Strassen theorem involving more general cone of functions.
\end{rem}

\section{Regularity of the optimal map}\label{sec:regularity}

\subsection{Key convexity estimate for the regularity}

The following result studies the solutions of the weak optimal transport problem given by Theorem \ref{thm:main-WOT}.

\begin{prop} \label{prop:lemme_technique}
Let $(E,d)$ be a complete separable $\CAT(\kappa)$ space with $\kappa \geq 0$, and let $\mu,\nu \in \mathcal{P}_2(E)$.
In the $\kappa>0$ case, we assume \eqref{eq:kappa>0}.
Let $T$ be the optimal map given by Theorem \ref{thm:main-WOT}.
There exists a subset $A \subset \operatorname{supp}\mu$ with $\mu(A)=1$ such that for all $x,y \in A$,
\begin{equation}
        \label{eq:technical_ineq}
       k d^2\bigl(T(x),T(y)\bigr) \leq d^2\bigl(x,T(y)\bigr) +d^2\bigl(y,T(x)\bigr) - d^2\bigl(x,T(x)\bigr)-d^2\bigl(y,T(y)\bigr),
\end{equation}
with $k=2$ if $\kappa = 0$ and $k = 2\sqrt{\ka}\ep/\tan(\sqrt{\ka}\ep)$ if $\ka >0$, where $\ep = \ep_{\mu,\nu}$.
\end{prop}

\begin{proof}
Let $p$ be an optimal kernel for $\mathcal{T}_c(\mu,\nu)$.
By Corollary \ref{cor:cyclical}, for any $x,y \in A$ (given in Corollary \ref{cor:cyclical}) and $q_1,q_2 \in \mathcal{P}_2(D)$ such that $q_1+q_2 = p_{x}+p_{y}$, it holds
\[
c(x,p_{x})+c(y,p_{y}) \leq c(x,q_1)+c(y,q_2).
\]
Taking $q_1 = (1-t) p_{x}+tp_{y}$ and $q_2 = (1-t) p_{y}+tp_{x}$, we get
\[
d^2\bigl(x,T(x)\bigr)+d^2\bigl(y,T(y)\bigr) \leq d^2\bigl(x,C(q_1)\bigr) +  d^2\bigl(y,C(q_2)\bigr).
\]
Let $[T(x),T(y)]_{t \in [0,1]}$ be the geodesic joining $T(x)$ to $T(y)$.
Following Lemma~\ref{rem:convex_mean_geodesic}, $[T(x),T(y)]_t$ is a convex mean of $q_1$, and $[T(y),T(x)]_t$ is a convex mean of $q_2$.
So we get
\[
d^2\bigl(x,T(x)\bigr)+d^2\bigl(y,T(y)\bigr) \leq d^2\bigl(x, [T(x),T(y)]_t\bigr) +  d^2\bigl(y, [T(y),T(x)]_t\bigr).
\]
In the case $\ka>0$, Lemma \ref{lem:Ohta} together with \eqref{eq:kappa>0} yields that the function $d^2(x,\,\cdot\,)$ is $k$-convex on $D$ for any $x$ in the support of $\mu$, with $k=2\sqrt{\ka}\ep/\tan(\sqrt{\ka}\ep)$.
For $\ka=0$, this holds with $k=2$.
Thus we have
\begin{multline*}
d^2\bigl(x, [T(x),T(y)]_t\bigr) +  d^2\bigl(y, [T(y),T(x)]_t\bigr)\\ \leq (1-t)d^2\bigl(x,T(x)\bigr)+td^2\bigl(x,T(y)\bigr) + (1-t)d^2\bigl(y,T(y)\bigr)+td^2\bigl(y,T(x)\bigr) - kt(1-t)d^2\bigl(T(x),T(y)\bigr).
\end{multline*}
Putting everything together, we get
\[
k(1-t)d^2\bigl(T(x),T(y)\bigr) \leq d^2\bigl(x,T(y)\bigr) +d^2\bigl(y,T(x)\bigr) - d^2\bigl(x,T(x)\bigr)-d^2\bigl(y,T(y)\bigr).
\]
Letting $t \to 0$ yields the claimed inequality \eqref{eq:technical_ineq}.
\end{proof}

\subsection{Non-positive curvature spaces}

We are now ready to complete the proof of Theorem \ref{thm:BS-NPC}.

\proof[Proof of Theorem \ref{thm:BS-NPC}]
The existence and uniqueness of $\bmu$ and $T$ follow from Theorem \ref{thm:main-WOT}.
It remains to prove that $T$ is $1$-Lipschitz.
By Proposition \ref{prop:lemme_technique} with $D = E$, it holds
\[
  2 d^2\bigl(T(x),T(y)\bigr) \leq d^2\bigl(x,T(y)\bigr) +d^2\bigl(y,T(x)\bigr) - d^2\bigl(x,T(x)\bigr)-d^2\bigl(y,T(y)\bigr),
\]
for $x,y \in A$, with $\mu(A) = 1$.
According to Reshetnyak’s quadruple comparison (see \cite[Proposition 2.4]{SturmNPC}), we have
\[
d^2\bigl(x,T(y)\bigr) +d^2\bigl(y,T(x)\bigr) - d^2\bigl(x,T(x)\bigr)-d^2\bigl(y,T(y)\bigr) \leq 2 d(x,y)d\bigl(T(x),T(y)\bigr).
\]
This yields
\[
d\bigl(T(x),T(y)\bigr) \leq d(x,y),
\]
for  all $x,y \in A$.
Since $\mu(A)=1$, $A$ is dense inside the support of $\mu$.
Therefore, $T$ admits a unique $1$-Lipschitz version defined on the support of $\mu$ which completes the proof.
\endproof

Beyond the existence of the projection $\bar\mu$, the proof of Theorem \ref{thm:BS-NPC} sheds light on the structure of the optimal weak transport plan from $\mu$ to $\nu$.

\begin{rem}\label{rem:martingale}
As in the Euclidean case, the optimal weak transport plan from $\mu$ to $\nu$ (see Theorem \ref{thm:projection_WOT}) decomposes into a deterministic $1$-Lipschitz first step, namely the transport map $T$ from $\mu$ to $\bar\mu$ given by Theorem \ref{thm:BS-NPC}, followed by a second non-deterministic step coupling $\bar\mu$ and $\nu$.
More precisely, since $\bar{\mu} \lc \nu$, by applying Theorem \ref{thm:strassen}, one obtains a kernel $(q_y)_{y\in E}$ such that $\bar\mu q = \nu$ and $\delta_y \lc q_y$ for $\bar\mu$-almost every $y$, which is precisely the definition of a martingale coupling in the sense of Émery and Mokobodzki \cite{EM91}.
Such couplings are in general not unique.
Since $y$ is only a convex mean of $q_y$ and need not be its barycenter, this coupling is moreover not always a two time-step martingale in the (stronger) sense of Sturm \cite{Sturm02}.
\end{rem}

As consequences of Theorem \ref{thm:BS-NPC}, the convex order projection enjoys two properties of independent interest: a Pythagorean-type inequality and a non-expansiveness property (generalizing \cite[Proposition 2.2]{Alfonsi-Jourdain-preprint}).

\begin{prop}\label{prop:projection_properties}
Let $(E,d)$ be a complete separable $\CAT(0)$ space and let $\nu \in \mathcal{P}_2(E)$.
For $\mu \in \mathcal{P}_2(E)$, denote by $\bar\mu$ the convex order projection of $\mu$ to $\{\eta \in \mathcal{P}_2(E) : \eta \lc \nu\}$ given by Theorem \ref{thm:BS-NPC}.
Then we have the following.
\begin{enumerate}
    \item $W_2^2(\mu,\nu) \geq W_2^2(\mu,\bar\mu) + W_2^2(\bar\mu,\nu)$ for every $\mu \in \mathcal{P}_2(E)$.
    \item $W_2(\bar\mu_1,\bar\mu_2) \leq W_2(\mu_1,\mu_2)$ for every $\mu_1,\mu_2 \in \mathcal{P}_2(E)$.
\end{enumerate}
\end{prop}

\begin{proof}
Let $(\Omega,\mathcal{A},\mathbb{P})$ be a non-atomic Polish probability space, and consider the space $L^2(\Omega,E)$ of square integrable $E$-valued random variables, equipped with the metric $D(X,Y) := \sqrt{\mathbb{E}[d^2(X,Y)]}$.
According to \cite[Proposition 3.10]{SturmNPC}, $(L^2(\Omega,E),D)$ is itself a $\CAT(0)$ space, whose geodesics are given pointwise by those of $E$:
the geodesic from $X_0$ to $X_1$ is $X_t(\omega) = [X_0(\omega),X_1(\omega)]_t$.
Since $\Omega$ is non-atomic, $W_2^2(\mu_1,\mu_2)$ coincides with $\min\{D^2(X_1,X_2) : \mathrm{Law}(X_i) = \mu_i\}$, the minimum being attained.

The set $\widetilde K := \{Z\in L^2(\Omega,E) : \mathrm{Law}(Z) \lc \nu\}$ is convex by testing convex functions pointwise along geodesics.
It is also closed, since by Lemma \ref{lem:tech}, the convex order can be tested against only Lipschitz functions bounded from below.
Let $T$ denote the optimal $1$-Lipschitz transport map from $\mu$ to $\bar\mu$ provided by Theorem \ref{thm:BS-NPC}.
Then for every random variable $X\in L^2(\Omega,E)$ of law $\mu$, the random variable $\bar X := T\circ X$ is the metric projection of $X$ to $\widetilde K$.
Indeed, $\bar X \in \widetilde K$ with $D^2(X,\bar X) = W_2^2(\mu,\bar\mu)$, while any $Z\in \widetilde K$ satisfies $D^2(X,Z) \geq W_2^2(\mu,\mathrm{Law}(Z)) \geq W_2^2(\mu,\bar\mu)$.

Now suppose that $(X,Y)$ gives an optimal coupling of $\mu$ and $\nu$.
The variable $Y$ lies in $\widetilde K$, so we can apply the Pythagorean-type inequality for metric projections in $\CAT(0)$ spaces (see \cite[Theorem 2.1.12(ii)]{Bacak}), which implies
\[ W_2^2(\mu,\nu) = D^2(X,Y) \geq D^2(X,\bar X) + D^2(\bar X, Y) \geq W_2^2(\mu,\bar  \mu) + W_2^2(\bar \mu, \nu).
\]
This shows item (1).

Item (2) follows similarly from the non-expansiveness of metric projections in $\CAT(0)$ spaces (see \cite[Theorem 2.1.12(iii)]{Bacak}) applied to two random variables  $X_1,X_2$ of laws $\mu_1,\mu_2$ realizing $D(X_1,X_2) = W_2(\mu_1,\mu_2)$.
Namely, denoting by $\bar X_i$ the metric projection of $X_i$ to $\widetilde K$, we have
\[
W_2(\bar\mu_1,\bar\mu_2)
 \leq D(\bar X_1,\bar X_2)
 \leq D(X_1,X_2).
 \qedhere \]
\end{proof}

\subsection{Spaces with curvature bounded from above}

In the $\kappa >0$ setting, Theorem \ref{thm:BS-NPC} still admits a partial extension, provided that the measures are sufficiently close.
One still obtains continuity of the optimal transport map, but only with $1/2$-H\"older regularity instead of $1$-Lipschitz regularity.

\begin{thm}\label{thm:BS-CATkappa}
Let $(E,d)$ be a complete separable $\CAT(\kappa)$ space with $\kappa > 0$, and let $\mu,\nu \in \mathcal{P}_2(E)$.
We assume \eqref{eq:kappa>0}.
Then the optimal transport map $T$ from $\mu$ to $\bar{\mu}$ given by Theorem \ref{thm:main-WOT} is H\"older continuous with exponent $1/2$.
More precisely, for all $x,y$ in the support of $\mu$,
\[
d\bigl(T(x),T(y)\bigr) \leq \sqrt{\frac{2\tan(\sqrt{\ka}\ep)}{\sqrt{\ka}}}\, d(x,y)^{1/2},
\]
with $\ep=\ep_{\mu,\nu}$.
\end{thm}

\begin{proof}
According to Proposition \ref{prop:lemme_technique}, for any two points $x_1,x_2 \in A$,
\[
k d^2\bigl(T(x_1),T(x_2)\bigr)
\leq
d^2\bigl(x_1,T(x_2)\bigr)+d^2\bigl(x_2,T(x_1)\bigr)
-d^2\bigl(x_1,T(x_1)\bigr)-d^2\bigl(x_2,T(x_2)\bigr),
\]
with $k=2\sqrt{\ka}\ep/\tan(\sqrt{\ka}\ep)$.
Using the identity $a^2-b^2=(a-b)(a+b)$ on the right hand side, we obtain
\[
k d^2\bigl(T(x_1),T(x_2)\bigr)
\leq
\sum_{i=1}^2
\bigl[d\bigl(x_{3-i},T(x_i)\bigr)-d\bigl(x_i,T(x_i)\bigr)\bigr]
\bigl[d\bigl(x_{3-i},T(x_i)\bigr)+d\bigl(x_i,T(x_i)\bigr)\bigr].
\]
By the triangle inequality, for both $i=1,2$,
\[
d\bigl(x_{3-i},T(x_i)\bigr)-d\bigl(x_i,T(x_i)\bigr) \le d(x_1,x_2),
\]
and by \eqref{eq:kappa>0},
\[
d\bigl(x_{3-i},T(x_i)\bigr)+d\bigl(x_i,T(x_i)\bigr) \le 2\ep.
\]
Therefore
\[
k d^2\bigl(T(x_1),T(x_2)\bigr)
\le 4\ep\,d(x_1,x_2).
\]
Since $k = 2\sqrt{\ka}\ep/\tan(\sqrt{\ka}\ep)$, this yields
\[
d^2\bigl(T(x_1),T(x_2)\bigr) \le \frac{2\tan(\sqrt{\ka}\ep)}{\sqrt{\ka}}\,d(x_1,x_2).
\]
Since $A$ is dense in the support of $\mu$, we conclude that $T$ admits a H\"older continuous version on $\operatorname{supp}\mu$.
\end{proof}

\section{A Strassen type result for barycentric martingales}\label{sec:barycentric}

Let $(E,d)$ be a $\CAT(0)$ space.
We assume in this section that $(E,d)$ is proper (i.e.\ closed balls in $E$ are compact), thus complete and separable.
Given two probability measures $\mu,\nu \in \mathcal{P}_2(E)$, it is well known that the existence of a martingale in the sense of \'Emery and Mokobodzki \cite{EM91} is equivalent to the convex domination of $\mu$ by $\nu$.
More precisely, according to e.g.\ Theorem \ref{thm:strassen}, $\mu\lc \nu$ if and only if there exists a pair of random variables $(M_0,M_1)$ with $M_0\sim \mu$ and $M_1\sim \nu$ such that for $\mu$-almost all $x \in E$, $x\in C(\mathrm{Law}(M_1 | M_0=x))$, where we recall that $C(p)$ denotes the set of convex means of a probability measure $p$ on $E$.

The aim of this section is to give a similar characterization for the existence of a martingale in the sense of Sturm \cite{Sturm02, SturmC}.
More precisely, we would like to characterize the set of couples $(\mu,\nu)$, $\mu,\nu \in \mathcal{P}_2(E)$, such that there exists a \emph{barycentric martingale} $(M_0,M_1)$ with $M_0\sim \mu$ and $M_1 \sim \nu$, that is a pair of random variables such that
\[
M_0 = \E[M_1|M_0] \qquad \text{a.s.},
\]
which means that $M_0 = b(\mathrm{Law}(M_1|M_0))$ a.s., where $b(p)$ denotes the barycenter of a probability measure $p\in \mathcal{P}_2(E)$ (recall Section \ref{sec:bary}).

For a measurable function $f : E \to \R \cup \{ +\infty \}$ with $f \geq -a(1+ d^2(x_0,\,\cdot\,))$ for some $a\geq 0$ and $x_0 \in E$, let us define its \emph{barycentric envelope} $\mathrm{bar}(f) : E \to \R\cup \{\pm \infty\}$ by
\[
\mathrm{bar}(f)(x) := \inf \left\{\int_E f\,dp : p \in \mathcal{P}_2(E),\, b(p) = x\right\}.
\]

\begin{remark}
Note that the function $\mathrm{bar}(f)$ is different from the convex envelope $\overline{f}$ considered in Section \ref{sec:envelope}.
It always holds $\overline{f} \leq \mathrm{bar}(f) \leq f$.
The function $\mathrm{bar}(f)$ is not convex in general, while $\overline{f}$ always is (recall Lemma \ref{lem:cvx-envelope}).
\end{remark}

\begin{remark}
Let $f:E \to \R \cup\{+\infty\}$ be a l.s.c.\ function with $f \geq -a(1+ d^2(x_0,\,\cdot\,))$ for some $a\geq 0$ and $x_0 \in E$.
Then $f$ is convex if and only if $\mathrm{bar}(f)=f$.
Indeed, $\mathrm{bar}(f)=f$ exactly means that
\[
f\bigl(b(p)\bigr) \leq \int_E f\,dp.
\]
Taking $p = (1-t) \delta_{x_0} + t \delta_{x_1}$ yields $f(x_t) \leq (1-t) f(x_0)+tf(x_1)$, with $(x_t)_{t\in [0,1]}$ the geodesic from $x_0$ to $x_1$.
Thus $f$ is convex.
Conversely, if $f$ is convex, then it satisfies Jensen's inequality \eqref{eq:Jensen} which amounts to $\mathrm{bar}(f)=f$.
\end{remark}

\begin{lem}
Let $f$ be a l.s.c.\ function, bounded from below and with $f\leq a(1+d^2(x_0,\,\cdot\,))$ for some $a \ge 0$ and $x_0\in E$.
Then $\mathrm{bar}(f)$ is measurable.
\end{lem}

\proof
Assume, without loss of generality, that $0 \leq f\leq a(1+d^2(x_0,\,\cdot\,))$.
For each $\ep>0$, define
\[
\mathrm{bar}_\ep(f)(x) := \inf \left\{\int_E \bigl(f+\ep d^2(x_0,\,\cdot\,)\bigr) \,dp : p \in \mathcal{P}_2(E),\, b(p) = x\right\},\qquad x\in E.
\]
Then, it is easily seen  that $\mathrm{bar}(f)=\inf_{k\geq 1}\mathrm{bar}_{1/k}(f)$.
Let us show that, for all $\ep>0$, $\mathrm{bar}_\ep(f)$ is l.s.c., this will imply in particular that $\mathrm{bar}(f)$ is measurable.

Let $(x_n)$ be a sequence converging to $x \in E$, and $p_n\in \mathcal{P}_2(E)$ be such that $b(p_n)=x_n$ and
\[
\mathrm{bar}_\ep(f)(x_n) \geq \int_E \bigl(f+\ep d^2(x_0,\,\cdot\,)\bigr) \,dp_n - r_n
\]
for some $r_n>0$ converging to $0$.
Let $(n_k)$ be a sequence of integers such that  $\mathrm{bar}_\ep(f)(x_{n_k}) \to \liminf_{n\to \infty}\mathrm{bar}_\ep(f)(x_n)$.
Since
\[
\mathrm{bar}_\ep(f)(x_{n_k}) \leq f(x_{n_k})+\ep d^2(x_0,x_{n_k}) \leq a\bigl(1+d^2(x_0,x_{n_k})\bigr) +\ep d^2(x_0,x_{n_k}),
\]
we conclude that
\begin{equation}\label{eq:uniform-integrability}
    \sup_{k\geq 1}\int_E d^2(x_0,y)\,p_{n_k}(dy)<+\infty.
\end{equation}
Therefore, since $E$ is proper, it follows from Prokhorov's theorem that the sequence $(p_{n_k})$ is pre-compact for the weak topology. Using again \eqref{eq:uniform-integrability} and \cite[Theorem 7.12]{Villani}, we see that any converging subsequence of $(p_{n_k})$ is also converging for the $W_1$ distance.
Thus, extracting a subsequence if necessary, we can assume that $p_{n_k} \to p$ for $W_1$.
Then, by the continuity of the barycenter for the $W_1$ distance (which follows from \eqref{eq:W1contraction}), we have $b(p_{n_k}) \to b(p)$, and so $b(p)=x$ and
\[
\lim_{k \to \infty} \mathrm{bar}_\ep(f)(x_{n_k}) \geq \int_E \bigl(f+\ep d^2(x_0,\,\cdot\,)\bigr) \,dp
\]
where the inequality follows from the fact that, for any l.s.c.\ non-negative function $g$, the map $p \mapsto \int_E g \,dp$ is l.s.c.\ for the weak topology on $\mathcal{P}(E)$.
In particular, we see that $p \in \mathcal{P}_2(E)$, and hence
\[
\int_E \bigl(f+\ep d^2(x_0,\,\cdot\,)\bigr) \,dp \geq \mathrm{bar}_\ep(f)(x)
\]
from which follows that
\[
\liminf_{n \to \infty} \mathrm{bar}_\ep(f)(x_n) \geq \mathrm{bar}_\ep(f)(x),
\]
thus $\mathrm{bar}_\ep(f)$ is l.s.c.
\endproof

\begin{thm}\label{thm:Strassenbar}
Let $(E,d)$ be a proper $\CAT(0)$ space and $\mu,\nu \in \mathcal{P}_2(E)$.
Then the following are equivalent.
\begin{enumerate}
\item There exists a barycentric martingale $(M_0,M_1)$ such that $M_0 \sim \mu$ and $M_1 \sim \nu$.
\item For every l.s.c.\ function $f:E \to \R$, bounded from below and with $f\leq a(1+d^2(x_0,\,\cdot\,))$ for some $a>0$ and $x_0\in E$, it holds
\[
\int_E \mathrm{bar}(f)\,d\mu \leq \int_E f\,d\nu.
\]
\end{enumerate}
\end{thm}

\begin{proof}
$(1) \Rightarrow (2)$.
Let $(M_0,M_1)$ be a barycentric martingale with marginals $\mu,\nu$.
Denoting $p_x = \mathrm{Law}(M_1 | M_0=x)$, we get
\[
\E[f(M_1) | M_0=x] = \int_E f\,dp_x \geq \mathrm{bar}(f)(x),
\]
so $\E[f(M_1) | M_0] \geq \mathrm{bar}(f)(M_0)$ a.s.\ and thus $\E[f(M_1)] \geq \E[\mathrm{bar}(f)(M_0)]$.

$(2) \Rightarrow (1)$.
For $p \in \mathcal{P}_2(E)$, define its \emph{variance} as
\[
V(p) := \inf_{x\in E}\int_E d^2(x,y)\,p(dy) = \int_E d^2\bigl(b(p),y\bigr)\,p(dy).
\]
We consider the weak optimal transport problem between $\mu$ and $\nu$ associated to the cost function
\[
C(x,p) := \int_E d^2(x,y)\,p(dy) - V(p) \ge 0.
\]
We set
\[
\mathcal{T}_{C}(\mu,\nu) := \inf_{p} \int_E C (x,p_x)\,\mu(dx),
\]
where the infimum runs over the set of probability kernels $p=(p_x)_{x\in E}$ such that $\mu p = \nu$.
For fixed $z \in E$, the map $p\mapsto \int_E d^2(z,y)\,p(dy)$ is linear, so $V$ is concave (with respect to convex interpolations) as the infimum of linear functions.
The cost function $C(x,p)$ is therefore convex in $p$, and is also continuous on $E\times \mathcal{P}_2(E)$ equipped with the product topology.
Indeed, the continuity of $(x,p)\mapsto \int_E d^2(x,y)\,p(dy)$ is classical, and $V$ satisfies
\[
\bigl|\sqrt{V(p)}-\sqrt{V(q)}\bigr| \leq W_2(p,q),
\]
by Minkowski's inequality.
According to \cite[Theorems 1.2 and 1.3]{BBP19}, the infimum in $\mathcal{T}_C(\mu,\nu)$ is attained, and it holds
\begin{equation}\label{eq:dualityhatT}
\mathcal{T}_C(\mu,\nu) = \sup_{f} \left\{ \int_E R_C f\,d\mu - \int_E f\,d\nu \right\},
\end{equation}
where the supremum runs over continuous functions $f$ bounded from below with $f\leq a(1+d^2(x_0,\,\cdot\,))$ for some $a\geq0$ and $x_0\in E$, and
\[
R_C f(x) := \inf_{p\in\mathcal{P}_2(E)} \left\{ \int_E f\,dp + C(x,p) \right\}.
\]
For $p\in\mathcal{P}_2(E)$ with $b(p) = x$, we have $C(x,p) = 0$, thereby
\[
R_C f(x) \leq \mathrm{bar}(f)(x).
\]
Combining this with the assumption (2), we get
\[
\int_E R_C f\,d\mu \leq \int_E \mathrm{bar}(f)\,d\mu \leq \int_E f\,d\nu,
\]
and so \eqref{eq:dualityhatT} yields $\mathcal{T}_C(\mu,\nu) \leq 0$.
Since $C$ is non-negative, we obtain $\mathcal{T}_C(\mu,\nu) =0$.
If $p$ is the primal optimizer, we thus get
\begin{equation}\label{eq:zero}
\int_{E} C(x,p_x)\,\mu(dx) = 0.
\end{equation}
By the uniqueness of barycenter, $C(x,p) = 0$ if and only if $b(p) = x$, thus \eqref{eq:zero} ensures that $b(p_x) = x$ for $\mu$-almost every $x$.
The coupling $\pi(dxdy) = p_x(dy)\mu(dx)$ then provides a barycentric martingale with marginals $\mu$ and $\nu$, which completes the proof.
\end{proof}

\appendix \section{Measurability of the transport map}\label{sec:measurablility}

\begin{lem}[Joint continuity of metric projections]\label{lem:joint-projection}
Let $E$ be a complete $\CAT(\kappa)$ space with $\kappa \ge 0$ and $D\subset E$ be a closed convex subset.
When $\kappa>0$, we take $U\subset E$ such that
\[
  \sup_{x\in U} \sup_{y\in D} d(x,y)<D_\kappa/2,
\]
and set $U=E$ when $\kappa=0$.
For any nonempty closed convex subsets $K_n,K$ of $D$ and $x_n,x\in U$, if $x_n\to x$ and $d_H(K_n,K)\to 0$, then we have $P_{K_n}x_n\longrightarrow P_Kx$, where $P_Kx$ denotes the metric projection of $x$ to $K$.
\end{lem}

\begin{proof}
Put $z_n=P_{K_n}x_n$, $z=P_Kx$, and $h_n=d_H(K_n,K)$. Choose $y_n\in K_n$ and $w_n\in K$ such that
\[
 d(y_n,z)\leq d(z,K_n)+n^{-1} \leq h_n+n^{-1},
 \qquad d(w_n,z_n)\leq d(z_n,K)+n^{-1} \leq h_n+n^{-1}.
\]
By assumption, we have $y_n \to z$ and $d(w_n,z_n) \to 0$.
By the choice of $z_n$,
\[
 d(x_n,z_n)\leq d(x_n,y_n),
\]
so that $\limsup_{n \to \infty} d(x_n,z_n)\leq d(x,z)$.
On the other hand,
\[
 d(x,z)\leq d(x,w_n)
 \leq d(x,x_n)+d(x_n,z_n)+h_n+n^{-1},
\]
and hence $d(x_n,z_n)\to d(x,z)$ as well as $d(x,w_n)\to d(x,z)$.

On the relevant common ball, by Lemma \ref{lem:Ohta}, the functions $d^2(x,\,\cdot\,)$ are
uniformly $k$-convex for some $k>0$. The variational inequality for the projection to $K$ (see \cite[Lemma 2.8]{KuwaeJensen}) gives
\[
  d^2(x,w_n)-d^2(x,z)
  \geq \frac{k}{2}d^2(w_n,z).
\]
Thus $w_n\to z$, and
\[
 d(z_n,z)\leq d(z_n,w_n)+d(w_n,z)\longrightarrow0.
\qedhere \]
\end{proof}

\section*{Declaration on the use of generative AI}

During the preparation of this work, the authors used ChatGPT 5.5 (OpenAI) in the following two ways.
First, as an assistance for exploring some of the arguments of the paper.
The most significant instance is Proposition~\ref{prop:stability}: we had only conjectured the statement of item 2, and its complete proof, going through item 1, was suggested to us by ChatGPT 5.5.
Second, for proofreading the manuscript.
All statements and proofs obtained with this assistance were subsequently checked and written up by the authors, who reviewed and edited the content as needed and take full responsibility for the content of this article.

\end{document}